\documentclass[11pt]{amsart}
\usepackage{amssymb, amsmath, amsfonts, amsthm, esint,graphicx,xcolor}
\usepackage[abbrev]{amsrefs}
\usepackage[margin=0.8 in]{geometry}
\usepackage[colorlinks=true, linkcolor=blue]{hyperref}

\def \eps {\varepsilon}

\def \hsp {\hspace{0.2 in}}

\DeclareMathOperator{\Ric}{Ric}

\DeclareMathOperator{\Hess}{Hess}
\DeclareMathOperator{\diam}{diam}
\DeclareMathOperator{\vol}{vol}

\DeclareMathOperator{\diver}{div}

\DeclareMathOperator{\SecFun}{II}

\DeclareMathOperator{\II}{II}
\DeclareMathOperator{\SN}{SN}

\newcommand{\lb}{\overline{\lambda}}

\title{Sharp Spectral Gap Estimates on Manifolds under Integral Ricci Curvature Bounds}
\author{Xavier Ramos Oliv\'e}
\address{Department of Mathematics and Computer Science, College of the Holy Cross, Worcester, MA 01610}
\email{\href{mailto:xramos-olive@holycross.edu}{ xramos-olive@holycross.edu}}
\thanks{X.R.O. is partially supported by an AMS-Simons PUI Grant.}
\author{Shoo Seto}
\address{Department of Mathematics\\
         California State University, Fullerton\\
         Fullerton, CA 92831}
\email{\href{mailto:shoseto@fullerton.edu}{shoseto@fullerton.edu}}
\thanks{S.S. is partially supported by an AMS-Simons PUI Grant.}
\author{Malik Tuerkoen}
\address{Department of Mathematics\\
       University of California, Irvine\\
         Irvine, CA 92697}
\email{\href{mailto:mtuerkoe@uci.edu}{mtuerkoe@uci.edu}}
\thanks{M.T. is partially supported by an AMS-Simons Travel Grant}
\subjclass{58C40, 35P15}
\date{}

\theoremstyle{definition} 
\newtheorem{theorem}{Theorem}[section]
\newtheorem{definition}[theorem]{Definition}

\newtheorem{lemma}[theorem]{Lemma}
\newtheorem{remark}[theorem]{Remark}

\newtheorem{proposition}[theorem]{Proposition}
\newtheorem{corollary}[theorem]{Corollary}

\newtheorem{thmx}{Theorem}

\begin{document}
\begin{abstract}
    We prove sharp spectral gap estimates on compact manifolds with integral curvature bounds. We generalize the results assuming a lower Ricci curvature bounds to the integral curvature setting. In doing so, we confirm a conjecture raised in \cite{ramos2020zhong}. This completes the picture of lower bounds on the spectrum of the Laplacian, assuming an integral curvature bound.
\end{abstract}
\maketitle
\vspace{-0.5 cm}

\section{Introduction}
Let $M$ be an $n$-dimensional compact Riemannian manifold, possibly with nonempty convex and $C^2$ boundary.  We will consider the eigenvalue problem of the Laplacian on $M$,
\begin{align*}
    \Delta u +\lambda u =0 \quad \textup{in }M,
\end{align*}
 imposing Neumann boundary conditions (vanishing of the normal derivative) in the case of nonempty boundary. It is easily seen that the eigenvalues are nonnegative and that $\lambda_0=0$ is simple. By spectral theory of compact self-adjoint operators, there exists a sequence of eigenvalues 
\begin{align*}
    \lambda_0 = 0 < \lambda _1 \leq \lambda_2 \leq \dots \rightarrow +\infty.
\end{align*}
The study of the eigenvalues (or spectrum) has been an active area of research in geometric analysis and here we will focus on the first nonzero eigenvalue $\lambda_1$, in particular, on obtaining a sharp quantitative lower bound in terms of geometric data.  Lichnerowicz \cite{lichnerowicz1958geometrie} showed that for compact manifolds with a positive Ricci curvature lower bound  the first nonzero eigenvalue is bounded below by the first nonzero eigenvalue on an $n$-sphere with radius matching the Ricci curvature bound, and Zhong and Yang \cite{Zhong1984ONTE} for compact manifolds with nonnegative Ricci curvature, the first nonzero eigenvalue is bounded below by $\frac{\pi^2}{D^2}$ where $D$ is the diameter of the manifold. These results can be summarized as
\begin{equation*}
    \lambda_1(M) \geq \begin{cases}
        \lambda_1(S^n(1/\sqrt{K}))=nK & \text{ for } \Ric\geq(n-1)K>0,\\
        \lambda_1(S^1(D/\pi))=\frac{\pi^2}{D^2} &\text{ for }\Ric \geq 0.
    \end{cases}
\end{equation*}
The above two estimates have been unified by Bakry and Qian \cite{bakry2000some} (see also the earlier work by Kröger \cite{kroger1992spectral}) in the following form:
 \begin{thmx} \label{thm: bakryqian}
    Let $M$ be an $n$-dimensional compact Riemannian manifold (with possibly non-empty convex and $C^2$ boundary), $\Ric\geq (n-1)K$, $K\in \mathbb{R}$ and $ \diam(M)\leq D$.  Let $\lambda_1(M)$ be the first nonzero eigenvalue of the Laplacian.  Then
    \begin{align}\label{ineq: kroger estimate}
        \lambda_1(M)\geq \lambda_1(n,K,D),
    \end{align}
    where $\lambda_1(n,K,D)$ is the first nonzero eigenvalue of the one-dimensional eigenvalue problem
    \begin{equation}\label{model: kroger neumann}
        \begin{cases}
            w''-T_{n , K}w'+\lambda w=0 &\textup{ on } [-\frac{D}{2}, \frac{D}{2}]\\
            w'(\pm\frac{D}{2})=0,
        \end{cases}
    \end{equation}
    where $T_{n,K}$ is given by 
    \begin{align}\label{def: of T}
        T_{n,K}(x) = \begin{cases}(n-1)\sqrt{K}\tan\left(\sqrt{K}x\right), \quad\,\,\,\,\,&\textup{ if }K >0,\\
        0,\quad\quad\,\quad\quad\quad\quad \quad\quad\,&\textup{ if }K =0,\\-(n-1)\sqrt{-K}\tanh\left(\sqrt{-K}x\right),& \textup{ if }K <0.
        \end{cases}
    \end{align}
\end{thmx}
Andrews and Clutterbuck proved the same estimate using a two-point maximum principle \cite{andrews2013sharp}.
\begin{remark}
    The above theorem unifies the results of Lichnerowicz and Zhong-Yang as it can be directly computed that $\lambda_1(n,0,D)=\frac{\pi^2}{D^2}$ and $\lambda_1(n,K,\frac{\pi}{\sqrt{K}})=nK$ when $K>0$.  In fact, we have an almost interpolation given by Shi and Zhang \cite{shi2007lower}:
    \begin{align*}
        \lambda_1(n,K,D) \geq \max_{s \in (0,1)} 4(s-s^2)\frac{\pi^2}{D^2}+s(n-1)K.
    \end{align*}
    For the case $K < 0,$ Yang \cite{yang1990_estimate_eigenvalue} derived the lower bound \begin{align}
        \label{ineq: yang estimate} \lambda_1(M) 
    \geq \frac{\pi^2}{D^2}\exp\left(-c_nD\sqrt{(n-1)|K|}\right)
    \end{align} where $c_n =\max\{2, n-1\}.$ 
\end{remark}

In this paper we generalize Theorem \ref{thm: bakryqian} to the setting of integral Ricci curvature.  To this end, we let $\rho(x) $ be the smallest eigenvalue of the Ricci tensor at $x\in M$ and for a constant $K \in \mathbb R,$ we let $\rho_K(x)$ be the amount of Ricci curvature below $(n-1)K$ at $x$, that is, 
\begin{align}\label{eq: def of rho k}
 \rho_K(x) = \max\left\{-(\rho(x)-(n-1)K), 0\right\}.
\end{align}
We measure the amount of Ricci curvature below $(n-1)K$ in an $L^p$ sense with the following quantity 
\begin{align}\label{eq: def of k bar}
    \kappa(p,K) = \left(\frac{1}{\vol(M)}\int_M\rho_K^p\, d\textup{V}\right)^{\tfrac{1}{p}}.
\end{align}
Note that $\kappa(p,K)=0$ if and only if $\Ric\geq (n-1)K$.  Many classical results of geometric analysis under a pointwise Ricci lower bound have been generalized to integral curvature, c.f. \cite{petersen1997relative}, and for the weighted case in \cite{wu}. The Lichnerowicz estimate with control on integral curvature has been established by Aubry \cite{aubry2007finiteness}. In fact, Aubry showed that for $K>0$ 
\begin{align}\label{ineq: aubry}
    \lambda _1(M) \geq nK\left(1-C(n,p)\kappa(p,K)\right).
\end{align}
The Zhong-Yang estimate (i.e. the case $\Ric \geq 0$) has been generalized to the integral curvature setting. In fact, this was established by the first two authors along with G. Wei and Q.S. Zhang \cite{ramos2020zhong}. C. Rose and G. Wei also generalized the Zhong-Yang estimate to the closely related Kato-type condition \cite{rose-wei}. In \cite{ramos2020zhong} it was conjectured that an integral curvature version of \eqref{ineq: kroger estimate} holds. In the present article, we confirm this conjecture and generalize the results of Kr\"oger and Bakry-Qian to the integral curvature setting. This completes the picture of lower bounds of eigenvalues under integral Ricci curvature conditions. Our main theorem reads as follows.
\begin{theorem}\label{thm: main theorem}
Let $M$ be an $n$-dimensional compact Riemannian manifold (possibly with non-empty convex and $C^2$ boundary), $n\geq 2$, and $\diam(M)\leq D$. Let $\lambda_1(M)$ be the first non-trivial eigenvalue of the Laplacian and let $p > \tfrac{n}{2}$. Then for any $\alpha \in (0,1)$ there exists $\varepsilon_0 = \varepsilon_0(n,p,\alpha,D)>0$ such that whenever $\kappa(p,K) < \varepsilon_0,$ one has that
\begin{align*}
\lambda_1(M)\geq  \alpha{\lambda_1}(n,K,D),
\end{align*}
where $\lambda_1(n,K,D)$ is the first non-trivial eigenvalue of the one-dimensional eigenvalue problem \eqref{model: kroger neumann}.
\end{theorem}
To prove this theorem one shows  a gradient comparison, that is, we aim to prove an estimate of the type 
\[|\nabla u|^2 \leq (w')^2\circ w^{-1}(u(x)),\]
where $u$ is the eigenfunction associated with $\lambda_1(M)$ and $w$ is the eigenfunction associated to $\lambda_1(n,K,D).$  This approach was introduced by Kröger \cite{kroger1992spectral} and also used by Bakry and Qian \cite{bakry2000some}. Valtorta \cite{VALTORTA20124974} extended this to the $p$-Laplacian framework (see also the work by Naber and Valtorta in \cite{naber2014sharp}). We extend the gradient comparison to the integral curvature setting. Our approach uses the technique introduced by Zhang and Zhu in \cite{zhang2017li}, which was successfully applied in \cite{ramos2020zhong}. The key idea here is to introduce an auxiliary function $J$ that absorbs critical terms in the maximum principle calculation (see Definition \ref{def: of J}). One of the key difficulties in our proof here is that one has to perturb parameters of the one-dimensional model to obtain a gradient comparison. These parameters will be close to the parameters that are used in the pointwise lower bound case. Finally, we will be able to obtain a comparison to the eigenvalue of the model \eqref{model: kroger neumann} with the \textit{correct} parameters since $\lambda_1(n,K,D)$ is a continuous function with respect to $n, K,$ and $D$, i.e., our estimate is sharp in the sense that it recovers Theorem \ref{thm: bakryqian} in the limit where ${\rm Ric} \geq (n-1)K$. Note that for the case $K<0, $ under the assumptions of Theorem \ref{thm: main theorem}, we get an integral curvature version of  \eqref{ineq: yang estimate}. In fact, for any $\alpha \in (0,1),$ $p> \tfrac{n}{2}, $ there exists $\varepsilon >0$ such that whenever $\kappa(p,K) < \varepsilon$ then one has that
\begin{equation*}
\lambda_1 \geq \alpha\frac{\pi^2}{D^2}\exp\left(-c_nD\sqrt{(n-1)|K|}\right).
\end{equation*}

\subsection{Overview of the Paper} In Section \ref{sec: not and prelims}, we fix our notation and recall some results from previous work that will be of significance throughout this article.
In Section \ref{sec: grad comparison}, we prove the key theorem, namely a gradient comparison theorem. The main difficulty there is to circumvent the need for a lower bound on the curvature, which we overcome by introducing the auxiliary function $J$ which absorbs the integral curvature terms (see Lemma \ref{lem: J}). In Section \ref{sec: maximum comparison}, we show that there exists a one-dimensional model, whose maximum and minimum are the same as the first eigenfunction on the manifold. This is crucial to obtain a sharp spectral gap comparison. In Section \ref{sec: diameter comparison}, we give the proof of Theorem \ref{thm: main theorem}, via a diameter comparison along with the monotonicity properties of $\lambda_1(n,K,D)$ as a function of $D.$

\subsection*{Acknowledgments}
The authors wish to thank Guofang Wei for helpful discussions.

\section{Notation and Preliminaries}\label{sec: not and prelims}
We first introduce some notations and mention important results of previous works that will be important to us throughout this work. We let $M$ be a $n$-dimensional manifold and denote $u$ to be eigenfunction for the first non-zero eigenvalue $\lambda_1,$ that is
\begin{align*}
    -\Delta u = \lambda_1 u \quad \textup{in }M,
\end{align*}
with Neumann boundary condition, in case $\partial M \neq \emptyset.$ We denote the second fundamental form of $\partial M$ by $\SecFun.$ %In Theorem \ref{thm: main theorem}, we assume that either $\partial M = \emptyset$ or that $\SecFun>0.$ 
Further, we assume that $u$ is normalized such that 
\begin{align*}
   -1 = \min u < 0< \max u := u^* \leq 1.
\end{align*}
We denote
 \[A_u := \Hess u\left(\frac{\nabla u}{|\nabla u|},\frac{\nabla u}{|\nabla u|}\right)\] to simplify notation in the proof of Theorem \ref{thm: gradient comp}.
\subsection{Properties of the one-dimensional models}
In the following, we introduce some notation for our perturbed parameters. More precisely, we will perturb the coefficients of $T_{n,K}$ so that we have more room in the maximum principle calculation (see Theorem \ref{thm: gradient comp}). We denote the perturbed parameters to be $\overline K < K,$ $\overline{n}>n,$ and $\overline \lambda > \lambda_1$ and we let $T$ be a solution to the following Ricatti equation
\begin{align}
    \label{eq: equation for T}T' = \frac{T^2}{\overline{n}-1}+(\overline{n}-1)\overline K.
\end{align}
  Note that depending on the sign of $\overline K,$ there are different possible solutions for $T.$ We now collect the solutions to the above equation that we will consider throughout this work. For the case $\overline K>0$ it suffices to consider only one solution to \eqref{eq: equation for T}, which is given in \eqref{def: of T}. That is, for the case $\overline K >0$ we  consider $T= T_{\overline{n},\overline K}$ to be defined 
\begin{align*}
    \label{eq: def of T for K >0} T_{\overline{n},\overline K}(t) = (\overline{n}-1) \sqrt{\overline K} \tan(\sqrt{\overline K}t), \quad \textup{defined on }\left(-\frac{\pi}{2\sqrt{\overline K}}, \frac{\pi}{2\sqrt{\overline K}}\right).
\end{align*}The case $
\overline K <0$ is more delicate, we consider the following two solutions to \eqref{eq: equation for T}.
\begin{align*}
      T_{\overline{n},\overline K,+}(t) =& -(\overline{n}-1)\sqrt{-\overline K} \textup{coth}(\sqrt{-\overline K}t), \quad \textup{for }t \in (0,\infty) ,\\
  T_{\overline{n},\overline K}(t)=& -(\overline{n}-1)\sqrt{-\overline K} \textup{tanh}(\sqrt{-\overline K}t),\quad \textup{for }t \in \mathbb R.
\end{align*}
 In this work, we focus on the cases $\overline K >0$ and $\overline K<0,$ since $\overline K$ is a perturbation of $K,$ we can choose it possibly smaller and still get a sharp estimate simply by approximation. Note that the result for $K =0$ was also shown in \cite{ramos2020zhong}.

 In our comparison, we will consider a one-dimensional model where $\overline \lambda$ and the left endpoint are fixed:
\begin{definition}\label{def: of one dim model}
    For $\overline \lambda >0,$ $\overline{n} > 1$ and $\overline K\in \mathbb R\backslash\{0\}$ fixed, let $T$ be one of the solutions above, defined on the corresponding interval $I$ indicated above and let $w = w^{\overline\lambda}_{T, a}$ be the solution to the initial value problem on $I,$ where $a \in \overline I.$
    \begin{equation}\label{model: kroger}
        \begin{cases}
            w''-Tw'+\overline\lambda w=0 \\
            w(a) = -1, \, w'(a)=0. 
        \end{cases}
    \end{equation}
    We also let ${d}(a,T, \overline\lambda)>0$ be the smallest positive number $ d>a$ such that $w'(a+{d})=0$ and set ${d}(a,T, \overline\lambda) = \infty$ if such a number does not exist. We sometimes omit the dependence $ d$ on $\lambda,$ as long as there is no confusion.
    We define the right end point of the interval to be $b = b(a,T,\overline \lambda),$ that is
    \begin{align*}
        b := a + d(a,T,
        \overline\lambda). 
    \end{align*}
\end{definition}
\begin{remark}
    Existence and uniqueness of $w$ follow from standard ODE techniques.
\end{remark}

Note that if $ d = {d}(a,T, \overline \lambda)<\infty,$ in Definition \ref{def: of one dim model}, the number $\overline\lambda >0$ is then a \textit{Neumann eigenvalue} of the operator \[\mathcal L_T= \frac{d^2}{dt^2}-T\frac{d}{dt } \quad \textup{on }[a,b].\] 
 While in general, the Neumann eigenvalues do not satisfy the domain monotonicity property, the first Neumann eigenvalues of the operators $\mathcal L_T$ does satisfy the domain monotonicity property. That is, for intervals $I_1,I_2$ \begin{align}
     \label{ineq: domain monotonicity} I_1 \subset I_2 \textup{, then } \lambda_1(T,I_1) \geq \lambda_1(T,I_2)
 \end{align} where $\lambda_1(T,I_i)$ denotes the first Neumann eigenvalue of the operator $\mathcal L_T$ on the interval $I_i$.

Denote $d_{\overline{n},\overline K,\overline \lambda}$ to be the length of the symmetric interval of the operator $\mathcal{L}_{T_{\overline{n},\overline K}}$ with eigenvalue $\overline \lambda,$ that is \[\overline \lambda = \lambda_1\left(T_{\overline{n},\overline K},\left[-\frac{d_{\overline{n},\overline K,\overline \lambda}}{2}, \frac{d_{\overline{n},\overline K,\overline\lambda}}{2}\right]\right).\] Note that in the case $\overline K >0,$ we assume that \[
\overline \lambda > \overline{n}\overline K,\]as otherwise such a $d$ might not exist (or $d = \tfrac{\pi}{\sqrt{\overline K}}$ for $\lambda = \overline{n}\overline K$).
 
\begin{proposition}[\cite{bakry2000some}*{Section 7, Theorem 13}]
\label{thm: central interval}
    For any $\overline{n}\geq 1,$ $\overline K\in \mathbb R,$ let $T$ be any solution to \eqref{eq: equation for T}. Then \[\lambda_1\left(T, [a,b]\right) \geq \lambda _1\left(T_{\bar{n},\bar{K}},\left[ -\tfrac{d}{2}, \tfrac{d}{2}\right]\right),\]
    where $d = b-a.$
\end{proposition}
As a consequence, combining Proposition \ref{thm: central interval} and \eqref{ineq: domain monotonicity}, we find that given $\overline{n}, \overline K,$ and $\overline \lambda >0,$ (if $\overline K>0$ assume $\overline \lambda>\overline{n}\overline K$)  \begin{align*}
    \label{ineq: length of interval comparison} d(  a,T,\overline  \lambda) \geq d_{\overline{n}, \overline K, \overline \lambda}.
\end{align*}
This will be crucial later in the proof of Theorem \ref{thm: main theorem}.
\subsection{The auxiliary function}
Next we introduce our auxiliary function $J$ that will play a key role in our computation for proving the gradient estimate in the integral curvature case.
\begin{definition}\label{def: of J}
    For $K \in \mathbb R,$ $\tau >1,$ and $\sigma \geq 0,$ let $J$ be a positive solution to the equation \begin{equation}\label{Jeq}
    \Delta J-\tau \frac{|\nabla J|^2}{J}-2J\rho_K=-\sigma J,
\end{equation}
where in the case $\partial M \neq \emptyset,$ we assume Neumann boundary conditions. Here $\rho_K$ is given as in \eqref{eq: def of rho k}.
\end{definition}
The intimate relationship between the integral curvature condition $\kappa(p,K) <\varepsilon$ (where $\kappa(p,K)$ is defined in \eqref{eq: def of k bar}) and $J$ becomes clear in the following lemma.
\begin{lemma}\label{lem: J} On a compact manifold $(M,g)$ (possibly with non-empty $C^2$ convex boundary) of dimension $n,$ diameter $D>0,$ and for any $\delta>0$, there exists $\eps = \eps(n,p,D,\tau)>0$ such that if $\kappa(p,K)\leq \eps$, then there is a number $\sigma$ and a corresponding function $J$ solving \eqref{Jeq} such that $0\leq \sigma \leq 4\eps$ and 
\begin{align*}
    |J-1|\leq \delta.
\end{align*}
\end{lemma}

\begin{remark}
The closed and $K=0$ case was proved in \cite{ramos2020zhong}. The same argument extends to the present setting, with $K\in\mathbb R$ and possibly nonempty convex boundary. Indeed, under the transformation
\[
J=W^{-\frac{1}{\tau-1}},
\]
equation \eqref{Jeq} is equivalent to
\[
\Delta W+VW=\widetilde{\sigma}W,
\]
where
\[
V=2(\tau-1)\rho_K,
\qquad
\widetilde{\sigma}=(\tau-1)\sigma,
\]
together with Neumann boundary conditions in the case $\partial M\neq\emptyset$. The proof then follows exactly as in \cite[Proposition~2]{ramos2020zhong}: one estimates $W$ from above and below by Moser iteration, which requires a rough Neumann Poincar\'e inequality and a global Neumann Sobolev inequality. We establish these two inequalities under integral Ricci curvature bounds in Appendix~\ref{appendix} for convenience, since we were unable to find a convenient reference for the exact form needed here. Since the Neumann condition eliminates boundary terms in the integration by parts, no further modification is required. For related Sobolev inequalities under integral curvature assumptions in the boundaryless, $K=0$, or local setting, see \cite{gallot1988isoperimetric, petersen1998integral, dai2018local,  WangWei, RamosSeto}.
\end{remark}

% \begin{remark}
%     The closed and $K=0$ case was proved in \cite{ramos2020zhong}.  By following the argument there, we can see that it holds for $K\in \mathbb{R}$.  The key observation is that under the transformation $J=W^{-\frac{1}{\tau-1}}$, \eqref{Jeq} is equivalent to the eigenvalue equation
%     \begin{align*}
%         \Delta W+VW=\tilde{\sigma}W,
%     \end{align*}
%     where $V=2(\tau-1)\rho_K$ and $\tilde{\sigma}=(\tau-1)\sigma$.\footnote{In the case $\partial M \neq \emptyset,$ one has that $W$ satisfies Neumann boundary condition.}  This form allows us to estimate $W$ from above and below using Poincar\'e inequality and Sobolev inequality.  Such tools were established under a general integral curvature setting by Gallot \cite{gallot1988isoperimetric}. See also Petersen-Sprouse \cite{petersen1998integral} or Dai-Wei-Zhang  \cite{dai2018local}.
    
% \end{remark}

\section{Gradient Comparison}\label{sec: grad comparison}
In this section we establish the key gradient comparison of the first eigenfunction on the manifold to the first eigenfunction of the one-dimensional model. %In the following, we let $u$ be the eigenfunction on $M$ with first non-trivial eigenvalue $\lambda_1$, that is 
%The main obstacle compared to the setting of Kr\"oger is the lack of point-wise control of the Ricci curvature. To overcome this difficulty, we introduce the auxiliary function $J$ and perturb the coefficients of the one-dimensional model slightly. 
\begin{theorem}\label{thm: gradient comp}
 For every $\delta> 0$ there exist parameters $\tau>1,$ $\overline{n} >n$, and a constant $\eps_0=\varepsilon_0 (n,p,D, \delta)>0$ with the following property.  If $\kappa(p,K)<\eps_0$, let $J$ be the positive solution of \eqref{Jeq}, with $|J-1|\leq \delta$ and associated constant $\sigma$.  Define $\overline{K}$ by 
 \begin{align*}
   \overline K =
   \begin{cases}
       (1-\delta)K\frac{ n-1}{\overline{n}-1}-(1+\delta)\frac{\sigma}{\overline{n}-1}\quad \textup{if }K \geq 0 \\
       (1+\delta)K\frac{ n-1}{\overline{n}-1}-(1+\delta)\frac{\sigma}{\overline{n}-1}\quad \textup{if }K <0.
       \end{cases}
\end{align*}
 Let $T$ be a solution of 
 \begin{align*}
     T'=\frac{T^2}{\overline{n}-1}+(\overline{n}-1)\overline{K},
 \end{align*}
and let $w = w^{\overline \lambda}_{T,a}$ be the one-dimensional initial value problem \eqref{model: kroger} on an interval $[a,b]$ such that $w' \geq 0$ on  $[a,b],$ where 
 \[\overline{\lambda}=(1+2\delta)\lambda_1.\] Assume that 
 \[[-1,u^\star] \subset [-1, w(b)].\]
 Then 
\begin{align*}
    J|\nabla u|^2 \leq (w')^2\circ (w^{-1}(u)),\label{ineq: gradient-comparison}
\end{align*}
where $J$  is given by Definition \ref{def: of J}.
\end{theorem}
\begin{remark}
The constants $\overline{n}$ and $\overline K$ are explicit constants that are defined in the proof.  They are small perturbations of $n$ and $K,$ depending on $\delta.$
\end{remark}
\begin{proof}
We first consider the case $\partial M = \emptyset.$  By contradiction, assume that \[J|\nabla u|^2> (w'(w^{-1}(u(x))))^2\] at some point $x\in M$.  For $c>1$, define \[Q = J|\nabla u|^2-(cw')^2((cw)^{-1}(u(x))),\]   where we choose $c$ such that at the maximal point $\overline x,$ we have $Q=0,$ i.e.  
\begin{align}\label{eqn: at x}
J(\overline x)|\nabla u|^2(\overline x) = (cw')^2((cw)^{-1}(u(\overline x))).
\end{align}
At that point $\overline x$, we have 
\begin{align}\label{first-order}
\nabla Q(\overline x)&=0 \\ \label{second-order}
\Delta Q(\overline x)&\leq 0.
\end{align}
It is easy to see that at $\overline x,$ \eqref{first-order} implies that \begin{align}
    \Hess u( \nabla u, \cdot) &= c \frac{w''}{J}\nabla u -\frac{|\nabla u|^2}{2J}\nabla J,\label{eqn: first-order-condition-matrix} \\
   A_u &= c \frac{w''}{J} -\frac{1}{2J}\langle\nabla J ,\nabla u\rangle, \label{eqn: first-order-condition}
\end{align}
where we write $w'' := w''((cw)^{-1}(u(\overline x)))$, and similarly for $w, w'$ and $w'''$ to simplify the notation. % Let 
Moreover, \eqref{eqn: at x}, \eqref{second-order} together with \eqref{eqn: first-order-condition-matrix} gives
\begin{align}
0 &\geq \frac{1}{2}(\Delta J)|\nabla u|^2+\langle \nabla J,\nabla |\nabla u|^2\rangle +\frac{1}{2}J\Delta |\nabla u|^2-\frac{w'''}{w'}|\nabla u|^2-cw''\Delta u \nonumber\\
&= \frac{1}{2}(\Delta J)|\nabla u|^2+2\frac{cw''}{J}\langle \nabla J,\nabla  u\rangle  -\frac{|\nabla J|^2}{J}|\nabla u|^2+\frac{1}{2}J\Delta |\nabla u|^2-c^2w'w'''J^{-1}+c\lambda_1 w'' u.\label{ineq: max-principle-step1}
\end{align}
To continue, we will apply the Bochner formula, together with \[\Ric \geq -\rho_K+(n-1)K\] and the eigenvalue equation \[\Delta u = -\lambda_1 u,\] so that  \begin{align}
\begin{split}
\frac{1}{2}\Delta |\nabla u|^2 &= |\Hess u|^2+\Ric(\nabla u,\nabla u)+\langle \nabla u, \nabla \Delta u\rangle\\
&\geq |\Hess u|^2+\left(-\rho_K+(n-1)K-\lambda_1\right)|\nabla u|^2. \label{ineq: Bochner-estimate}
\end{split}
\end{align} To estimate $|\Hess u|^2,$ we use the refined Cauchy-Schwarz inequality of the Hessian:
\begin{align}
\begin{split}
|\Hess u|^2 &\geq \frac{(\Delta u)^2}{n}+\frac{n}{n-1}\left(\frac{\Hess u(\nabla u,\nabla u)}{|\nabla u|^2} - \frac{(\Delta u)}{n}\right)^2\\
&=\frac{\lambda_1^2 u^2}{n}+\frac{n}{n-1}\left(A_u+\frac{\lambda_1 u}{n}\right)^2\\
&=\frac{\lambda_1^2 u^2}{n-1}+\frac{n}{n-1}A_u^2+\frac{2\lambda_1 u}{n-1}A_u.\label{ineq: hessian-lower-bound}
\end{split}
\end{align}
Using \eqref{ineq: Bochner-estimate} and \eqref{ineq: hessian-lower-bound} in \eqref{ineq: max-principle-step1}, we get 
\begin{align*}
0  \geq &\frac{1}{2}(\Delta J)|\nabla u|^2+\frac{2cw''}{J}\langle \nabla u,\nabla J\rangle -\frac{|\nabla J|^2}{J}|\nabla u|^2
+J\frac{\lambda_1^2 u^2}{n-1}+J\frac{n}{n-1}A_u^2+J\frac{2\lambda_1 u}{n-1}A_u\nonumber\\
& +\left(-\rho_K+(n-1)K-\lambda_1\right)J |\nabla u|^2-c^2w'w'''J^{-1}+c\lambda_1 w''u
\end{align*}
Applying the first order condition \eqref{eqn: first-order-condition} we have
\begin{align}
\begin{split}
0  \geq &\frac{1}{2}(\Delta J)|\nabla u|^2+\frac{2cw''}{J}\langle \nabla u,\nabla J\rangle -\frac{|\nabla J|^2}{J}|\nabla u|^2
+J\frac{\lambda_1^2 u^2}{n-1}+J\frac{n}{n-1}\left(c \frac{w''}{J} -\frac{1}{2J}\langle\nabla J ,\nabla u\rangle\right)^2\\
& +J\frac{2\lambda_1 u}{n-1}\left(c \frac{w''}{J} -\frac{1}{2J}\langle\nabla J ,\nabla u\rangle\right)-\rho_KJ |\nabla u|^2 +\left((n-1)K-\lambda_1\right)(cw')^2-c^2w'w'''J^{-1}+c\lambda_1 w''u.\label{ineq: hessian-lower-bound-2}
\end{split}
\end{align}
At $t = (cw)^{-1}(u(\overline x)),$ this becomes, after rearranging
\begin{align*}\begin{split}
0  \geq &\left(\frac{1}{2}(\Delta J)-\frac{|\nabla J|^2}{J}-\rho_KJ\right)|\nabla u|^2+\frac{2cw''}{J}\langle \nabla u,\nabla J\rangle 
+J\frac{\lambda_1^2 c^2w^2}{n-1}\\
&+J\frac{n}{n-1}\left(c \frac{w''}{J} -\frac{1}{2J}\langle\nabla J ,\nabla u\rangle\right)^2
 +J\frac{2\lambda_1 cw}{n-1}\left(c \frac{w''}{J} -\frac{1}{2J}\langle\nabla J ,\nabla u\rangle\right)\\
 &+\left((n-1)K-\lambda_1\right)(cw')^2-c^2w'w'''J^{-1}+c^2\lambda_1 w''w.
 \end{split}
\end{align*}
Rewriting, we have
\begin{align}
\begin{split}
    0 &\geq \frac{1}{2}\left((\Delta J)-2\frac{|\nabla J|^2}{J}-2\rho_KJ\right)|\nabla u|^2+\frac{(n-2)cw''}{(n-1)J}\langle \nabla u,\nabla J\rangle+{\frac{n}{4J(n-1)}\langle\nabla J ,\nabla u\rangle^2}\\
    &\,+\frac{(n+1)\lambda_1 c^2w''w}{n-1}+\frac{J\lambda_1^2(cw)^2}{n-1} -\frac{\lambda_1 c w }{n-1}\langle\nabla J ,\nabla u\rangle\\
    &+\left((n-1)K-\lambda_1\right)(cw')^2+c^2J^{-1}\left(\frac{n}{n-1}(w'')^2-w'w'''\right).\label{ineq: inequality from second derivatives rearranged 1}
    \end{split}
\end{align}
We now bound the mixed terms $\langle \nabla J,\nabla u\rangle$. We let $\alpha, \beta >0$ and get that 
\begin{align}
    -\frac{\lambda_1 cw}{n-1}\langle \nabla J,\nabla u\rangle &\geq -\alpha J\frac{\lambda_1^2 c^2w^2}{(n-1)}-\frac{|\nabla J|^2|\nabla u|^2}{\alpha 4 (n-1)J}\label{ineq: mixed terms 1}
\end{align}
and
\begin{align}
\frac{n-2}{n-1}\frac{cw''}{J}\langle \nabla J,\nabla u\rangle &\geq -\frac{n}{n-1}\beta \frac{c^2(w'')^2}{J}-\frac{(n-2)^2}{n(n-1)}\frac{|\nabla u|^2|\nabla J|^2}{\beta 4J}.\label{ineq: mixed terms 2}
\end{align}

Applying \eqref{ineq: mixed terms 1} and \eqref{ineq: mixed terms 2} to \eqref{ineq: inequality from second derivatives rearranged 1}, we get that 
\begin{align*}
0 &\geq\frac{1}{2}\left((\Delta J)-\left(2+\frac{(n-2)^2}{{2}\beta n(n-1)} +\frac{1}{2(n-1)\alpha}\right)\frac{|\nabla J|^2}{J}-2\rho_KJ\right)|\nabla u|^2\\
&+J(1-\alpha)\frac{\lambda_1^2 c^2w^2}{n-1}+\frac{n}{n-1}(1-\beta)\frac{c^2(w'')^2}{J}\\
& +\Bigl((n-1)K -\lambda_1\Bigr)(cw')^2-c^2\frac{w'''w'}{J}+\frac{n+1}{n-1}\lambda c^2w''w%\\
%&+\textcolor{blue}{\frac{n-\delta^{-1}}{n-1}\frac{(\langle \nabla J,\nabla u\rangle)^2 }{4J}}.
\end{align*}
We apply the equation satisfied by $J$ 
\begin{align*}
\frac{1}{2}\left((\Delta J)-\left(2+\frac{(n-2)^2}{{ 2}\beta n(n-1)} +
\frac{1}{2(n-1)\alpha}\right)\frac{|\nabla J|^2}{J}-2\rho_KJ\right)|\nabla u|^2=-\sigma J|\nabla u|^2
\end{align*}
and the fact that we are at a point $Q=0$ so that $J|\nabla u|^2=(cw')^2$. Then
\begin{align}
\begin{split}\label{mainineq}
0&\geq J(1-\alpha)\frac{\lambda_1^2 c^2w^2}{n-1}+\frac{n}{n-1}(1-\beta)\frac{c^2(w'')^2}{J}\\
& \hspace{0.2 in}+\Bigl((n-1)K-\lambda_1 -\sigma\Bigr)(cw')^2-c^2w'''w'J^{-1}+\frac{n+1}{n-1}\lambda_1 c^2w''w. 
\end{split}
\end{align}

Note that we have not used \eqref{model: kroger} up to this point.  Recall that $T$ satisfies \eqref{eq: equation for T}. Thus 
\begin{align*}
    w'''w' & =\left(\frac{\overline{n}}{\overline{n}-1}T^2 + (\overline{n}-1)\overline{K}\right)(w')^2  -\overline\lambda Tw'w -\overline\lambda (w')^2 ,\\
    (w'')^2 &= T^2(w')^2 -2\overline \lambda T w'w +\overline\lambda^2 w^2, \\
    w''w &= Tw'w -\overline\lambda w^2. 
\end{align*}
Using the above identities to \eqref{mainineq}, we re-write in terms of $\bar{\lambda}^2w^2$, $\bar{\lambda}Tw'w$ and $T^2(w')^2$ so that
\begin{align*}
0&\geq  \left((1-\alpha)\frac{J^2\lambda_1^2}{\overline{\lambda}^2}-(n+1)\frac{J\lambda_1}{\overline{\lambda}}+n-n\beta\right)\frac{\overline{\lambda}^2w^2}{J(n-1)}\\
&\hsp-\left(n+1-2n\beta-(n+1)\frac{J\lambda_1}{\overline{\lambda}}\right)\frac{\bar{\lambda}Tw'w}{J(n-1)}\\
&\hsp+\left(\frac{\overline{n}-n}{\overline{n}-1}-n\beta\right)\frac{T^2(w')^2}{J(n-1)}\\
&\hsp+\Bigl((n-1)K-\lambda_1 - \sigma-J^{-1}((\overline{n}-1)\bar{K}+\overline{\lambda})\Bigr)(w')^2.
\end{align*}
For convenience, set $y := \frac{J\lambda_1}{\lb}$. We observe that the first three lines form a quadratic form 
\begin{align*}
Q(u,v)=Pu^2-Suv+Rv^2,
\end{align*}
with  $u:=\bar{\lambda}w$, $v:=Tw'$ and 
\begin{align*}
    P&=(1-\alpha)y^2-(n+1)y+n(1-\beta)\\
    S&=n+1-2n\beta-y(n+1)\\
    R&=\frac{\overline{n}-n}{\overline{n}-1}-n\beta
\end{align*}
so that 
\begin{equation}\label{ineq: final ineq}
0\geq  \frac{1}{J(n-1)}Q(u,v)+\Bigl((n-1)K-\lambda_1 - \sigma-J^{-1}((\overline{n}-1)\bar{K}+\overline{\lambda})\Bigr)(w')^2.
\end{equation}
To arrive at a contradiction, we will first show that $Q$ is a positive semi-definite quadratic form and then show that the remaining term is positive. \\
\textbf{Step 1.} $Q$ is a positive semi-definite quadratic form. It suffices to show that 
\begin{align*}
P \geq 0, \quad 
R \geq 0, \quad 
4PR \geq S^2.
\end{align*}
Now, for any $\delta >0$, and by Lemma \ref{lem: J} there exists $\varepsilon >0$ such that if the $\kappa(p,K) <\varepsilon$ then $|J-1|<\delta.$ Choosing \[\overline \lambda = (1 + 2 \delta ) \lambda_1,\] we obtain 
\begin{align*}
\frac{1-\delta}{1+2\delta} < y < \frac{1+\delta}{1+2\delta}.
\end{align*}
For notational convenience, we let 
\begin{align*}
    x = 1 - y, \quad  \delta_1 = \frac{\delta}{1+2\delta}, \text{ and } \eta = \frac{(n+1)^2}{n-1}.
\end{align*}From the bounds on $y$, we have $x \in (\delta_1, 3\delta_1)$.  Now choose
\begin{align}\label{eq: choice-alpha-beta-n}
\alpha = \delta_1^2, \quad \beta = \frac{\delta_1^2}{n}, \quad 
\overline{n} &= \frac{n - \eta\delta_1}{1 - \eta\delta_1}
\end{align}
We claim that with this choice of $\alpha, \beta$ and $\overline{n},$ $Q$ is positive semi-definite for $\delta>0$ sufficiently small.  We rewrite $P = P(\alpha, \beta), S = S(\beta),$ and $R = R(\beta),$ exactly in terms of $x$ as function of $\alpha$ and $\beta$:
\begin{align*}
P = x(n-1+x) - \alpha(1-x)^2 - n\beta,\quad  S = (n+1)x - 2n\beta, \quad   
R = \frac{\overline{n}-n}{\overline{n}-1} - n\beta
\end{align*}

We first show non-negativity of $P$.  Since $ \delta_1< x <1$,
\begin{align*}
P \ge \delta_1(n-1) - \alpha - n\beta  = \delta_1(n - 1 - 2\delta_1)>0,
\end{align*}
where the last inequaltiy holds for  $0 < \delta < (n-1)/2$.
Next we show the non-negativity of $R$. By our choice of $\overline{n}$, we get
\begin{align*}
R & = \delta_1(\eta - \delta_1) \geq 0,
\end{align*}
for $\delta >0$ small.

Finally, we show non-negativity of the determinant term.  We expand $4PR - S^2$ so that
\begin{align*}
4PR - S^2 &= 4(P_0 - \alpha (1-x)^2 - n\beta)(R_0 - n\beta) - (S_0 - 2n\beta)^2 \\
&= (4P_0R_0 - S_0^2) + 4n\beta(S_0 - P_0) - 4(\alpha (1-x)^2 + n\beta)R_0 + 4(\alpha (1-x)^2 + n\beta)n\beta - 4n^2\beta^2.
\end{align*}
Denote  $P_0 = P(0,0)$, $S_0 = S(0),$ and $R_0 =R(0)$. 
Then \[
S_0 - P_0 = (n+1)x - x(n-1+x) = x(2-x) > 0, 
\] so that we can drop the strictly positive terms $4n\beta(S_0 - P_0)$ and $4(\alpha y^2 + n\beta)n\beta$ to obtain a lower bound. Using $0< x < 1$, we get:
\begin{align}\label{ineq: lower-bound-of-det}
4PR - S^2 \ge (4P_0R_0 - S_0^2) - 4(\alpha + n\beta)R_0 - 4n^2\beta^2.
\end{align}
We now bound the leading term $4P_0R_0 - S_0^2$. Substituting $R_0 = \eta\delta_1$:
\begin{align*}
4P_0R_0 - S_0^2 &= x \big[ 4(n-1)\eta\delta_1 + 4x\eta\delta_1 - (n+1)^2 x \big].
\end{align*}
Substitute $\eta = \frac{(n+1)^2}{n-1}$ into the first inner term, and use $x \le 3\delta_1$ to bound the negative term:
\begin{align*}
4P_0R_0 - S_0^2 &\ge x \big[ 4(n+1)^2\delta_1 - 3(n+1)^2\delta_1 \big] \\
&= x(n+1)^2\delta_1\\
&\geq  (n+1)^2\delta_1^2,
\end{align*}
where the last inequality follows from that fact that $x \ge \delta_1$. Substituting this back into \eqref{ineq: lower-bound-of-det}, together with our choices of $\alpha$ and $\beta$ in \eqref{eq: choice-alpha-beta-n} gives
\begin{align*}
4PR - S^2 &\ge (n+1)^2\delta_1^2 - 4(2\delta_1^2)(\eta\delta_1) - 4(\delta_1^2)^2 \\
&= \delta_1^2 \big[ (n+1)^2 - 8\eta\delta_1 - 4\delta_1^2 \big] >0,
\end{align*}
for $\delta >0$ sufficiently small. 
Thus $Q$ is positive semi definite.

\textbf{Step 2.} The remaining terms of \eqref{ineq: final ineq} are positive. To this end, we have 
\begin{align*}
    &(n-1)K-(\overline{n}-1)\frac{\overline K}{{J}} -\sigma+\frac{\overline \lambda}{J}-\lambda_1\\
   \geq & (n-1)K-(\overline{n}-1)\frac{\overline K}{J} -\sigma+\frac{\lambda_1(1+2\delta)}{1+\delta}-\lambda_1\\
   >& (n-1)K-(\overline{n}-1)\frac{\overline K}{J} -\sigma\geq0,
\end{align*}
whenever 
\[\overline K \leq JK\tfrac{ n-1}{\overline{n}-1}-J\tfrac{\sigma}{\overline{n}-1}.\]
Depending on the sign of $K,$ we choose 
\begin{align*}
   \overline K =
   \begin{cases}
       (1-\delta)K\frac{ n-1}{\overline{n}-1}-(1+\delta)\frac{\sigma}{\overline{n}-1}\quad \textup{if }K \geq 0 \\
       (1+\delta)K\frac{ n-1}{\overline{n}-1}-(1+\delta)\frac{\sigma}{\overline{n}-1}\quad \textup{if }K <0.
       \end{cases}
\end{align*}
This finishes the proof of Theorem \ref{thm: gradient comp}.
\end{proof}

\begin{remark}
    For the case $n \geq 3,$ one can show that the maximum principle proof works, without showing that the quadratic form is positive semi definite. 
\end{remark}
\begin{lemma}
   Assume that $\partial M \neq \emptyset$ and assume that $\SecFun\geq  0.$ If \[Q(x) = J|\nabla u(x)|^2-(cw')^2((cw)^{-1}(u(x))),\] achieves the maximum at a boundary point $\overline x \in \partial M,$ one has that the Equations \eqref{first-order} and Inequality \eqref{second-order} still hold true. 
\end{lemma}
\begin{proof}We first verify Equation \eqref{first-order}. Since $\overline x$ is a maximal point, we know that all derivatives tangential to $\partial M$ vanish and that the normal derivative of $Q$ is greater or equal to zero, that is we know that $\langle \nabla Q, n\rangle \geq 0.$ Our goal is to show that actually $\langle \nabla Q, n\rangle =0.$ Since $u$ satisfies Neumann boundary conditions, we know that \[\langle \nabla u, n\rangle =\langle \nabla J, n\rangle = 0\quad \textup{at} \quad \overline x.\]  For simplifying the notation, we set \[\psi(s) = (cw')^2((cw)^{-1}(s))\] calculate that \begin{align*}
    \langle \nabla Q, n\rangle = 2 J\Hess u(\nabla u , n ) + \langle \nabla J, n\rangle |\nabla u|^2+\psi'(u(x))\langle \nabla u, n\rangle   = -2J\II(\nabla u, \nabla u) \leq 0,
\end{align*}
where the last equality follows from the definition of the second fundamental form. It follows that \eqref{first-order} holds true. The inequality in \eqref{second-order} then follows in a straight forward way.
\end{proof}

\section{Maxima of Eigenfunctions}\label{sec: maximum comparison}
In this section, we show that given the eigenfunction $u$ on the manifold $M$ and eigenvalue $\lambda_1$ with \[-1 = \min u \quad \textup{and}\quad  u^\star := \max u  \leq 1,\] there exists $T = T_{\overline{n},\overline K}$ satisfying \eqref{eq: equation for T} and an interval $I$ and a Neumann eigenfunction $w$ to the eigenvalue \[\overline \lambda = (1+2\delta ) \lambda_1,\] such that 
\begin{align}
    \label{eqn: max match} \min_I w = \min u, \quad \max_I w = u^\star.
\end{align}
To show this, we first show a comparison result concerning the maxima of the eigenfunction $u$ and the one-dimensional model, which will be Theorem \ref{thm: maximum comparison}. In the subsequent part of this section, we will prove a minumum and maximum matching, namely we will show \eqref{eqn: max match}, which will be summarized in Theorem \ref{thm: maximum matching theorem}.
\subsection{Maximum Comparison} Here and in the following, let $\delta >0$ be fixed, and $\eps>0,$ $\overline{n}, \overline K, \overline \lambda$ be as in Theorem \ref{thm: gradient comp}. In the case $\overline K>0,$ we choose it to be possibly slightly smaller but still positive:
\begin{align}\label{eq: choice over overline K for K>0}
    \overline K = \min\left\{\frac{nK\left(1-C(n,p)\varepsilon\right)}{\overline{n}},(1-\delta)K\frac{ n-1}{\overline{n}-1}-(1+\delta)\frac{\sigma}{\overline{n}-1}\right\},
\end{align}
where the first term in the minimum is the lower bound of \eqref{ineq: aubry}. We distinguish the cases $\overline K>0$ and $\overline K<0$ and start by defining $m_{\overline{n}, \overline K}.$
Define 
\begin{align*}
\overline T_{\overline{n},\overline K}(t) = \begin{cases}
        {(\overline{n}-1)}\sqrt{\overline K}{\textup{tan}(\sqrt{\overline K} t)}, \quad &\textup{if }\overline K >0\\
          -{(\overline{n}-1)\sqrt{-\overline K}}{\textup{coth}(\sqrt{-\overline K}  t)}, \quad &\textup{if }\overline K <0.
    \end{cases}
\end{align*}
Note that for $\overline K>0$ we have that $ T_{\overline{n}, \overline K}= \overline T_{\overline{n}, \overline K}.$
We let $w = w^{\overline \lambda}_{\overline T_{\overline{n}, \overline K},a}$ to be the solution to the initial value problem \eqref{model: kroger} with $T = \overline T_{\overline{n},\overline K}$ , where we let $a = 0$ if $\overline K < 0$ and $a = -\tfrac{\pi}{2\sqrt{\overline K}}$ if $\overline K >0.$ We denote
\begin{align*}
    m_{\overline{n}, \overline K} := w(a + d(\overline T_{\overline{n}, \overline K},a))
\end{align*}
    the maximum value of $w$ on the interval $[a, a +  d(\overline T_{\overline{n},K}, a)].$ Recall that $b := a + d(\overline T_{\overline{n},K}, a).$
We start by showing \begin{align}
    \label{ineq: maximum comparison} u^\star \geq m_{\overline{n}, \overline K}.
\end{align} 
\begin{theorem}\label{thm: maximum comparison}
Let $p > \tfrac{n}{2},$ and $n \geq 2.$
For every $\delta >0,$ there exists $\varepsilon >0$ such that whenever $\kappa(p,K) < \varepsilon$ then \eqref{ineq: maximum comparison} holds.
\end{theorem}

 The proof of Theorem \ref{thm: maximum comparison} will be divided into several steps and follows the work of Bakry and Qian \cite{bakry2000some} (see also \cite{naber2014sharp}). However, since we are working under an integral curvature assumption, some parts of the proof need modification slightly. Through this first part of the section, we assume for contradiction that $\max u = u^\star < m_{\overline{n}, \overline K}.$ Let us fix some notation.

We let $t_0 \in \left(a,b\right)$ denote the unique zero of $w,$ where $w$ is chosen as above. Let $g = w^{-1}\circ u$ and define the measure $m$ on $[a , b ]$ by 
\begin{align*}
    m(A) := \textup{V}(g^{-1}(A)), \quad \textup{for any Borel measurable }A \subset [a, b],
\end{align*}
where we denote $\textup{V}$ to be the Riemannian measure on $M.$
This implies that for any bounded and measurable function $f$ on $[a, b],$ we get  
\begin{align*}
    \int_{a}^b f (s)\, dm(s)= \int_{{}M} f(g(x))\, d\textup{V}(x).
\end{align*}
\begin{proposition}\label{prop4.2}
Using the notation from above, assume that  $1-\delta< J < 1+\delta$, $u$ and $w$ satisfy the condition of Theorem \ref{thm: gradient comp}.  Then 
    the function \[E(s) := -\exp\left((1-\delta)\lambda_1\int _{t_0}^s\frac{w(t)}{w'(t)}\,dt\right)\int_a^sw(r)\, dm(r)\]
    is increasing on $(a,t_0]$ and decreasing on $[t_0, b).$
\end{proposition}
\begin{proof}
    Choose smooth nonnegative function $H(s)$ with compact support in $(a,b)$.  Define $G:[-1,w(b)]\to \mathbb{R}$ by
\begin{align*}
    \frac{d}{dt}[G(w(t))]=H(t), \quad G(-1)=0.
\end{align*}
Choose a function $F$ that satisfies $F(t)+tF'(t)=G(t)$.  Then
\begin{align*}
\Delta( uF(u)) &= \diver( (\nabla u)F+uF'\nabla u)\\
&=(\Delta u)F +2|\nabla u|^2F'+uF''|\nabla u|^2+uF'\Delta u\\
&= (\Delta u)(F+uF')+|\nabla u|^2(uF''+2F')\\
&=G(u)\Delta u + |\nabla u|^2G'(u).
\end{align*}
Since $\int_M\Delta(uF(u))d\textup{V}=0$,
\begin{align*}
\int_M (\Delta u)G(u)d\textup{V} = -\int_MG'(u)|\nabla u|^2d\textup{V}
\end{align*}
Apply the gradient comparison $|\nabla u|^2\leq J^{-1}(w'\circ w^{-1})^2(u)$ along with the eigenvalue equation so that
\begin{align*}
    \lambda_1 \int_M uG(u)d\textup{V} \leq \int_M J^{-1}G'(u)(w'\circ w^{-1})^2(u)d\textup{V} \leq \int_M (1-\delta)^{-1}G'(u)(w'\circ w^{-1})^2(u)d\textup{V},
\end{align*}
where the last inequality follows from the fact that $G'(u) \geq 0$ on $M.$
Since $w(g)=u$, we get 
\begin{align*}
    \lambda_1\int_M w(g)G(w(g))d\textup{V} \leq \int _M (1-\delta)^{-1}G'(u)(w'\circ w^{-1})^2(u)d\textup{V}.
\end{align*}
By definition of $m$ and since $a \leq g \leq b$,
\begin{align*}
    \lambda_1 \int_a^b v(s)G(v(s))dm(s)\leq \int_a^b (1-\delta)^{-1}G'(w(s))(w'(s))^2dm(s).
\end{align*}
Since $G'(w(s))(w'(s))^2=H(s)w'(s)$ and $G(w(s))=\int_a^sH(r)dr$, we have
\begin{align*}
    \int_a^b\left(\int_s^b\lambda_1 w(r)m(dr)\right)H(s)ds \leq \int_a^b(1-\delta)^{-1}H(s)w'(s)m(ds).
\end{align*}
Using the fact that $\int_Mu dV = 0$, we have $\int_a^b w(r)dm(r) = 0$ so that $\int_s^bw(r)\,dm(r) = -\int_a^sw(r)\,dm(r)$.  Let $A(s)=-\lambda_1 \int_a^sw(r)\,dm(r)$.  Then
\begin{align*}
    \int_a^b (1-\delta)^{-1}H(s)w'(s)dm(s)-H(s)A(s)\,ds \geq 0
\end{align*}
for any positive $H(s)$. 
Thus we obtain that 
\[ (1-\delta)^{-1}w'dm-A ds\geq 0.\]
Finally, we may rewrite this as follows 
\[ -\frac{(1-\delta)^{-1}}{\lambda_1}\frac{w'}{w}dA-Ads \geq 0, \]
which then, in-turn, implies that on $(a,t_0],$ since $ w\leq 0,$ $w'\geq 0$
we get that  
\[ dA+(1-\delta)\lambda_1\frac{w}{w'}Ads \geq 0. \]
This implies that the function 
\[E(s) =  A(s) \exp\left((1-\delta)\lambda_1\int_{t_0}^s\frac{w}{w'}\, dr\right)\]
is increasing on $(a,t_0]$ and decreasing on $[t_0,b).$
\end{proof}
Now with this proposition at hand, it is easy to show the following key proposition. 

Observe that the eigenvalue $\overline \lambda$ is different from $(1-\delta)\lambda_1$ (see Theorem \ref{thm: gradient comp}) which is why we have to modify the approach slightly. 
\begin{proposition}\label{prop:4.3}
Let $q = \frac{(1-\delta)\lambda_1}{\overline{\lambda}}$. For $\epsilon > 0$ sufficiently small, there exists a constant $C > 0$ independent of $\epsilon$ such that
$$ \int_{\{u \le -1+\epsilon\}} (-u) \, dV \le C \left( \int_{\{w \le -1+\epsilon\}} (-w) \, d\nu \right)^q, $$
where $\nu$ is the measure $d\nu(t) = \mu_{\overline n,\overline K}(t)dt$. Here $\mu_{\overline n,\overline K}(t)$ is the integrating factor satisfying $(\mu_{\overline n, \overline K}w')' = -\overline{\lambda}\mu_{\overline n, \overline K}w$.
\end{proposition}
Note that from our choice of $T_{\overline{n},\overline{K}}$, we have
\begin{align*}
\mu_{\overline{n},\overline{K}}(t)=\begin{cases}
        \cos^{\bar{n}-1}(\sqrt{\bar{K}}t), & \text{ if }\bar{K}> 0\\
        \sinh^{\bar{n}-1}(\sqrt{-\bar{K}}t), &\text{ if }\bar{K}<0
    \end{cases}
\end{align*}
which are both positive functions. 
\begin{proof}[Proof of Proposition \ref{prop:4.3}]
Let $A(s) = \int_a^s (-w(r)) \, dm(r) = \int_{\{u \le w(s)\}} (-u) \, dV$ and $B(s) = \int_a^s (-w(t)) \mu_{\overline n,\overline K}(t) \, dt$.
From the eigenvalue equation $(\mu_{\overline n,\overline K} w')' = -\overline{\lambda} \mu_{\overline n,\overline K} w$, we integrate from $a$ to $s$ to obtain:
$$ \mu_{\overline n,\overline K}(s)w'(s) = \overline{\lambda} B(s). $$
Taking the logarithmic derivative, we get
\begin{align*}
    \frac{(\mu_{\overline n,\overline K} w')'}{\mu_{\overline n, \overline K} w'} = \frac{-\overline{\lambda} \mu_{\overline n,\overline K} w}{\mu_{\overline n,\overline K} w'} = -\overline{\lambda} \frac{w}{w'}.
\end{align*}
Integrating from $s$ to $t_0$ and taking the exponent, we have
\begin{align*}\exp\left( \int_{t_0}^s \frac{w(t)}{w'(t)} \, dt \right) = \left( \frac{\mu_{\overline n,\overline K}(t_0) w'(t_0)}{\mu_{\overline n,\overline K}(s) w'(s)} \right)^{\frac{1}{\overline{\lambda}}} = \left( \frac{B(t_0)}{B(s)} \right)^{\frac{1}{\overline{\lambda}}}.
\end{align*}
By Proposition \ref{prop4.2}, the function
$$ E(s) = A(s) \exp\left( (1-\delta)\lambda_1 \int_{t_0}^s \frac{w(t)}{w'(t)} \, dt \right) $$
is increasing on $(a, t_0]$. Therefore, for $s \le t_0$, $E(s) \le E(t_0) = A(t_0)$. Substituting our expression for the exponential factor into $E(s)$ gives:
$$ A(s) \left( \frac{B(t_0)}{B(s)} \right)^{\frac{(1-\delta)\lambda_1}{\overline{\lambda}}} \le A(t_0). $$
Rearranging this, we get $A(s) \le A(t_0) B(t_0)^{-q} B(s)^q = C B(s)^q$, where $q = \frac{(1-\delta)\lambda_1}{\overline{\lambda}}$. Evaluating this at $s = w^{-1}(-1+\epsilon)$ finishes the proof.
\end{proof}

\begin{lemma}\label{lem: ball contained in preimage lemma}
The preimage $u^{-1}\left([-1,-1+\eps)\right)$ contains a ball of radius $r_{\varepsilon},$ which is determined by 
\begin{align*}
    r_{\varepsilon}=\sqrt{1-\delta}\left(w^{-1}\left(-1+{\varepsilon}\right)-a\right)
\end{align*} 
\begin{proof}
  Let $x_0$ be a minimum point of $u$, i.e., $u(x_0)=-1$.  Let $\bar{x}$ be another point on $M$.  Let $\gamma:[0,L]\to M$ be a unit speed minimizing geodesic from $x_0$ to $\bar{x}$, and define $f = u(\gamma(t))$.  Then by our gradient estimate we have
  \begin{align*}
      |f'(t)| = |\langle \nabla u(\gamma(t)),\gamma'(t)\rangle|\leq |\nabla u(\gamma(t))| \leq \frac{w'(w^{-1}(f(t)))}{\sqrt{J}}.
  \end{align*}
  From this we get (since $w' \geq 0$)
  \begin{align*}
      \frac{d}{dt}w^{-1}(f(t))\leq \frac{1}{\sqrt{(1-\delta)}},
  \end{align*}
  which implies 
  \[a\leq w^{-1}(f(t))\leq a + \frac{t}{\sqrt{(1-\delta)}}. \]
  Then since $w'$ is increasing in a neighborhood of $a,$ we can deduce that
  \begin{align*}
      w'(w^{-1}(f(t)))\leq w'\left(a + \frac{t}{\sqrt{(1-\delta)}}\right) \quad \textup{for }t \textup{ close to }0.
  \end{align*}
 From the above inequalities, and noting $f(0)=-1$, we have
  \begin{align*}
      |f(t)+1| \leq \frac{1}{\sqrt{1-\delta}}\int_{0}^tw'\left(a+\tfrac{s}{\sqrt{(1-\delta)}}\right)\, ds = w\left(a+\tfrac{t}{\sqrt{(1-\delta)}}\right)+1
  \end{align*}
  This leads us to deduce that if $t =d(x,x_0) <\sqrt{1-\delta}\left(w^{-1}\left(-1+{\varepsilon}\right)-a\right)$ then $u(x) < -1 +\varepsilon.$
\end{proof}
\end{lemma}
We are now in the position to prove the maximum comparison.

Recall also that $a=-\frac{\pi}{2\sqrt{K}}$ when $K>0$ and $a=0$ for $K\leq 0$
\begin{proof}[Proof of Theorem \ref{thm: maximum comparison}]
By contradiction, suppose that $\max u < m_{\overline{n},\overline K}$. Fix $\tau\in (a,t_0)$ sufficiently close to $a$, and set $h_\tau:=w(\tau)=-1+\eps_\tau$.  For $\tau$ close enough to $a$, we have $h_\tau<-\frac{1}{2}$. By Proposition \ref{prop:4.3}, we have:
\begin{align*}
    \vol(\{u \le h_\tau\}) \le -2 \int_{\{u \le h_\tau\}} u \, dV \le 2C' \left( \int_{\{w \le h_\tau\}} (-w) \, d\nu \right)^q \le 2C' \nu([a, \tau])^q
\end{align*}
From Lemma \ref{lem: ball contained in preimage lemma}, we infer that the set $\{u \le -1+\eps_\tau\}$ contains a ball $B(x_0, r_\tau)$ of radius $r_\tau = \sqrt{1-\delta}(\tau - a)$. Thus:
\begin{align*}\vol(B(x_0, r_\tau)) \le 2C' \nu([a, \tau])^q = 2C' \left( \int_a^\tau \mu_{\overline n,\overline K}(t) \, dt \right)^q.
\end{align*}
Near $t = a$, the density $\mu_{\overline{n},\overline K}(t)$ behaves like $(t-a)^{\overline n-1}$ regardless of the sign of $\overline K$. Thus, the integral is bounded by $C'' (\tau-a)^{\overline n}$, which gives:
$$ \vol(B(x_0, r_\tau)) \le C'' (\tau-a)^{\overline{n}q} = \tilde{C} r_\tau^{\overline{n}q}. $$

We now show that $\overline{n}q > n$ for $\delta$ sufficiently small. Recall that $\overline{\lambda} = (1+2\delta)\lambda_1$. Let $\delta_1 = \frac{\delta}{1+2\delta}$, which means $q = \frac{1-\delta}{1+2\delta} = 1 - 3\delta_1$. In Theorem \ref{thm: gradient comp}, we defined $\overline{n} = \frac{n - \eta\delta_1}{1 - \eta\delta_1}$ with $\eta = \frac{(n+1)^2}{n-1}$. For small $\delta_1 > 0$, we have:
$$ \overline{n}q = \frac{n - \eta\delta_1}{1 - \eta\delta_1} (1 - 3\delta_1) = \frac{n - (\eta+3n)\delta_1 + 3\eta\delta_1^2}{1 - \eta\delta_1}. $$
For $\overline{n}q > n$ to hold, we require \[n - (\eta+3n)\delta_1 + 3\eta\delta_1^2 > n - n\eta\delta_1\], which simplifies to $\eta(n-1) - 3n + 3\eta\delta_1 > 0$. Using our choice of $\eta$, we have \[\eta(n-1) - 3n = (n+1)^2 - 3n = n^2 - n + 1\]. Since $n \ge 2$, we have $n^2 - n + 1 \ge 3 > 0$. Therefore, $\overline{n}q > n$ holds strictly for all sufficiently small $\delta > 0$. 

The inequality $\vol(B(x_0, r_\tau)) \le \tilde{C} r_\tau^{\overline{n}q}$ with $\overline{n}q > n$ yields a contradiction to the fact that for an $n$-dimensional manifold (possibly with convex boundary), $\vol(B(x_0, r)) \ge c r^n$ for small $r$ (because $r^{\overline{n}q} \ll r^n$ as $r \to 0$). This proves that $u^* \ge m_{\overline{n},\overline K}$.
\end{proof}

\subsection{Maximum Matching}
We now show that \eqref{eqn: max match} holds. To infer this, we again distinguish the cases $\overline K >0$ and $\overline K<0$ in the proof below.  Recall that $a=-\frac{\pi}{2\sqrt{\bar{K}}}$ when $\bar K>0$ and $a=0$ when $\bar{K}<0$.

\begin{theorem}\label{thm: maximum matching theorem}Given $\delta >0,$ there exists $\varepsilon>0, \overline K$ and $\overline{n}$ (as in Theorem \ref{thm: gradient comp} and in \eqref{eq: choice over overline K for K>0}) such that whenever $\kappa(p,K) <\varepsilon$, there is $T$ satisfying \eqref{eq: equation for T}, an interval $I = [\bar{a},b(\bar a)]$ and a corresponding Neumann eigenfunction $w = w^{\overline \lambda}_{T,a}$ on $I$ solving the problem \eqref{model: kroger}, such that 
\begin{align}\label{eq: max and min matching}
    u^\star = \max_I w, \quad -1 = \min u= \min_I w. 
\end{align}
\end{theorem}

\begin{proof} 
We divide this proof into two cases.

\textit{Case 1:} Assume $\overline K >0.$ By our choice \eqref{eq: choice over overline K for K>0}, and in view of Aubry's estimate \eqref{ineq: aubry} we have that 
\begin{align}\label{ineq: eigenvalue comparison to Aubry}
    \overline \lambda > \overline{n}\overline K = \lambda_1\left(T_{\overline{n}, \overline K},\left[-\tfrac{\pi}{2\sqrt{\overline K}},\tfrac{\pi}{2\sqrt{\overline K}}\right]\right),
\end{align}
 where $\lambda_1\left(T_{\overline{n},\overline K},I_{\overline K}\right)$ is the first (non-trivial) Neumann eigenvalue of the operator $\mathcal L_{T_{\overline{n},\overline K}}$ on the interval \[I_{\overline K} : = \left[-\tfrac{\pi}{2\sqrt{\overline K}},\tfrac{\pi}{2\sqrt{\overline K}}\right].\]
 
If $b(a)\geq \frac{\pi}{2\sqrt{\overline K}}$, then by domain monotonicity, 
 \begin{align*}
     \lambda_1(T_{\bar{n},\bar{K}},I_{\bar{K}})\geq \lambda_1(T_{\bar{n},\bar{K}},[a,b(a)])=\overline \lambda,
 \end{align*}
 contradicting \eqref{ineq: eigenvalue comparison to Aubry}, hence $b(a)<\frac{\pi}{2\sqrt{\bar{K}}}$.

 If $u^\star = 1,$ we can choose the solutions given by \eqref{model: kroger} and choosing $T_{\overline{n},\overline K}$ as in \eqref{def: of T} and find an $-\tfrac{\pi}{2\sqrt{\overline K}}<\overline a<0$ such that the Neumann eigenvalue of the operator $\mathcal L_{T_{\overline{n},\overline K}}$ is equal to $\overline \lambda$ on a symmetric interval $[\overline a, -\overline a].$ Note that the corresponding eigenfunction $w$ is then an odd function and hence also satisfies $\max w = 1.$

  Next we assume that $u^\star<1.$ In view of Theorem \ref{thm: maximum comparison}, $u^\star<\frac{1}{u^\star}\leq \frac{1}{m_{\bar{n},\bar{K}}}$. Denote $w_a$ to be the solution starting at $a.$ Then 
    \begin{align*}
        w_{-}(x) = -\frac{w_a(-x)}{m_{\overline{n}, \overline K}}
    \end{align*}
    solves the initial value problem \eqref{model: kroger}, starting on $[-b,-a]$ and has maximum value $\tfrac{1}{m_{\overline{n}, \overline K}}.$  Note that $[-b,-a]\subset I_{\bar{K}}$. Since the maximum continuously depends on the starting point of the initial value problem even through the singularity points of $T_{\bar{n},\bar{K}}$, see (\cite{bakry2000some}, Section 3, Corollary 1), there exists an $\bar a \in I_{\overline K}$ and $\bar a< b(\bar a) \in I_{\overline K} $ such that the corresponding Neumann eigenfunction on $[\bar a,b(\bar a)]$ has Neumann eigenvalue $\overline \lambda$ and such that \eqref{eq: max and min matching} holds. This finishes the proof for $\overline K>0.$

    \textit{Case 2:} Assume $\overline K <0.$ 
    As in the case $
    \overline K >0,$ we know that if $u^\star = 1,$ we can find $\overline a <0$ such that $w = w^{\overline \lambda}_{T_{\overline{n},\overline K}, \overline a}$ is an odd function and Neumann eigenfunction on an interval $[-\overline a, \overline a].$
    
    To proceed with the proof, we distinguish two cases: $\overline \lambda \leq -\frac{(\overline{n}-1)^2\overline K}{4}$ and $\overline \lambda > -\frac{(\overline{n}-1)^2\overline K}{4}.$

    \textit{Case 2.1:} $\overline \lambda \leq -\frac{(\overline{n}-1)^2\overline K}{4}.$ In that case it follows from Proposition 28 in \cite{naber2014sharp} that there is an $a> \overline a$ such that the Neumann eigenfunction $w = w^{\overline \lambda}_{T_{\overline{n},\overline K}, a}$ on the interval $[a,b(a)]$ satisfies \eqref{eq: max and min matching}. 

    \textit{Case 2.2:} $\overline \lambda > -\frac{(\overline{n}-1)^2\overline K}{4}.$ From Proposition 32 in \cite{naber2014sharp}, we know that for any $u^\star \in [m_{\overline{n}, \overline K}, 1]$ \eqref{eq: max and min matching} holds. On the other hand, Theorem \ref{thm: maximum comparison} shows that $u^\star \geq m_{\overline{n}, \overline K}.$ Hence the proof is complete.
\end{proof}
\section{Proof of Theorem \ref{thm: main theorem}}\label{sec: diameter comparison}
In this section, we show the diameter comparison and obtain the spectral gap comparison.
\begin{proof}[Proof of Theorem \ref{thm: main theorem}]
    Let $u$ denote the first non-trivial eigenfunction on $M$ and $\lambda_1$ the corresponding eigenvalue. We assume that $u$ is scaled such that $-1 = \min u < \max u = u^\star \leq 1.$ By Theorem \ref{thm: maximum matching theorem}, we can find an interval $[a,b],$ a function $T = T_{\overline{n},\overline K}$ such that the corresponding Neumann eigenfunction $w$ (i.e. solving \eqref{model: kroger} with $w'(a) = w'(b) = 0$) on that interval satisfies 
    \[\max u = \max_{[a,b(a)]}w.\]
    Now consider a normalized, minimizing geodesic $\gamma: [0,l]\rightarrow M$ connecting the minimum point $x_0$ and the maximum point $y_0.$ Let $f(t) = u(\gamma(t))$ and choose $I \subset [0,l]$ in such a way that $I \subset f'^{-1}(0,\infty)$ and $f^{-1}$ is well-defined in a subset of full measure of $[-1, u^\star].$ Then, by the change of variables formula, we get 
    \begin{align*}
        D \geq \textup{diam}(M)  \geq \int_0^l dt \geq \int_{I}\,dt  &=\int_{-1}^{u^\star}\frac{dy}{f'(f^{-1}(y))} \\ &\geq \int_{-1}^{u^\star}\frac{dy}{\sqrt{1+\delta}w'(w^{-1}(y))}\\
       &=\frac{1}{\sqrt{1+\delta}} \int_{a}^{b(a)}1\,dt\\
       & = \frac{ d (T, a, \overline \lambda)}{\sqrt{1+\delta}}, 
    \end{align*}
    where as before $ d (T, a, \overline \lambda) = b(a)-a.$ By Proposition \ref{thm: central interval}, we know $ d (T, a, \overline \lambda)\geq d_{\overline{n},\overline  K, \overline \lambda}.$
By \eqref{ineq: domain monotonicity}  $d_{\overline{n},\overline  K, \overline \lambda}$ is decreasing in $\overline \lambda, $ which is why there exists $C_1 = C_1(\delta)>1$ such that $C_1(\delta)\rightarrow 1$ as $\delta\rightarrow 0^+$ and such that 
\begin{align*}
    \frac{ d_{\overline{n},\overline  K, \overline \lambda}}{\sqrt{1+\delta}} =   d_{\overline{n},\overline  K,C_1 \overline \lambda}.
\end{align*}
Note that since $\overline{n} \rightarrow n,$ $\overline K \rightarrow K$ as $\delta\rightarrow 0^+,$ we can choose $C_2(\delta)>0 $ such that yet $C_2(\delta) \rightarrow 1$
 as $\delta \rightarrow 0^+$ and 
 \begin{align*}
   d_{\overline{n},\overline  K,C_1 \overline \lambda} = d_{n,K,C_1 C_2 \overline \lambda}.
\end{align*}
Since $d_{\overline{n},\overline  K,C_1 \overline \lambda}$ is a strictly decreasing and continuous function in $\overline \lambda,$ we conclude that 
    \begin{align*}
     C_1C_2\overline \lambda \geq  \lambda_1( n,  K, D).
    \end{align*}
    In other words, since $\overline \lambda  = (1+2\delta) 
    \lambda_1,$ we have that for any $\alpha \in (0,1),$ there exists $\delta >0$ such that \begin{align*}
        \lambda_1 \geq \alpha  \lambda_1(n,K,D),
    \end{align*}
    as desired.
    Finally, note that if $\delta = 0,$ i.e. if $\kappa(p,K) =0,$ we recover the sharp estimate $\lambda_1 \geq \lambda_1(n,K,D)$.
\end{proof}

\newpage

\appendix
\section{A global Neumann Sobolev inequality under integral Ricci curvature bounds}\label{appendix}

In this appendix we prove the global Neumann Sobolev inequality, as well as a Neumann Poincar\'e inequality, which are needed for the Moser iteration used in the proof of Lemma~\ref{lem: J}. The argument follows the appendix of \cite{RamosSeto} in the Riemannian setting; see also \cite{gallot1988isoperimetric,petersen1998integral,WangWei,dai2018local} for related local Sobolev inequalities under integral curvature bounds.  The only new point is the treatment of the boundary; for this we assume that $\partial M$, if nonempty, is $C^2$ and convex. We use the averaged $L^q$ norm
\[
\|u\|_{q}^{*}:=\left(\fint_{M}|u|^q\,dv\right)^{1/q}.
\]

Our goal is to prove the following.

\begin{proposition}\label{prop:appendix-global-neumann-sobolev}
Let $p>\frac n2$. There exists $\eps=\eps(n,p,K,D)>0$ such that if $\kappa(p,K)\le \eps$, then there exists $C_S=C_S(n,p,K,D)$ such that for every $u\in W^{1,2}(M)$,
\begin{equation}\label{eq:appendix-target-sobolev-mean-zero}
\|u-u_M\|_{\frac{2p}{p-1}}^{*}
\le
C_S\,\|\nabla u\|_2^{*},
\qquad
u_M:=\fint_M u\,dv.
\end{equation}
Consequently,
\begin{equation}\label{eq:appendix-target-sobolev}
\|u\|_{\frac{2p}{p-1}}^{*}
\le
C_S\|\nabla u\|_{2}^{*}+\|u\|_{2}^{*}.
\end{equation}
\end{proposition}

The following standard result is key in being able to prove the Sobolev inequality when $\partial M\neq\emptyset$.

\begin{lemma}\label{lem:appendix-convex-geodesics}
If $x,y\in \operatorname{int}(M)$, then every minimizing geodesic joining $x$ and $y$ is contained in $\operatorname{int}(M)$.
\end{lemma}

\begin{proof}
Since $M$ is compact, it is complete as a metric space, hence any two points can be joined by a minimizing geodesic for the intrinsic distance. Let $\gamma:[0,\ell]\to M$ be such a minimizing geodesic joining $x$ and $y$.

Assume by contradiction that $\gamma(t_0)\in \partial M$ for some $t_0\in (0,\ell)$. Let $t_0$ be the first such time. Since $\gamma([0,t_0))\subset \operatorname{int}(M)$ and $\gamma$ is minimizing, the tangent vector $\gamma'(t_0)$ must be tangent to $\partial M$ at $\gamma(t_0)$. Convexity of $\partial M$ then implies that one can shorten the portion of $\gamma$ near $t_0$ by replacing it with a nearby curve lying in the interior of $M$, contradicting the minimizing property of $\gamma$.
\end{proof}

We first recall the pointwise Jacobian comparison and volume comparison estimates under integral Ricci curvature bounds; see \cite{petersen1997relative}. Let $V_K^n(r)$ denote the volume of the radius-$r$ ball in the simply connected $n$-dimensional space form of constant sectional curvature $K$.

\begin{lemma}\label{lem:appendix-local-jacobian}
Let $p>\frac n2$, and write the volume form in geodesic polar coordinates around
an interior point $x\in \operatorname{int}(M)$ as
\[
dv=\mathcal{A}(r,\theta)\,dr\,d\theta_{n-1}.
\]
Let
\[
dv_K=\mathcal{A}_K(r)\,dr\,d\theta_{n-1}
\]
denote the corresponding model volume form in the simply connected
$n$-manifold of constant sectional curvature $K$. Then for each fixed
$\theta\in S^{n-1}$,
\begin{equation}\label{eq:appendix-localvolcomp}
\frac{d}{dr}\left(\frac{\mathcal{A}(r,\theta)}{\mathcal{A}_K(r)}\right)
\le
\psi(r,\theta)\left(\frac{\mathcal{A}(r,\theta)}{\mathcal{A}_K(r)}\right),
\end{equation}
where $\psi$ satisfies
\begin{equation}\label{eq:appendix-psiestimate}
\int_0^r \psi(t,\theta)^{2p}\,\mathcal{A}(t,\theta)\,dt
\le
C(n,p)\int_0^r \rho_K(t,\theta)^p\,\mathcal{A}(t,\theta)\,dt.
\end{equation}
\end{lemma}

\begin{proof}
This follows from the proof of the relative volume comparison theorem in
\cite{petersen1997relative}; see also \cite[Lemma~A.2]{RamosSeto}.%Change citation
\end{proof}

\begin{theorem}\label{thm:appendix-vol-comp}
Let $p>\frac n2$. Then for $0<r\le R$ (and $R\le \pi/(2\sqrt K)$ if $K>0$), there exists $C=C(n,p,K,R)$ such that
\[
\left(\frac{\vol(B_x(R))}{V_K^n(R)}\right)^{\frac1{2p-1}}
-
\left(\frac{\vol(B_x(r))}{V_K^n(r)}\right)^{\frac1{2p-1}}
\le
C\,\kappa(p,K)^{\frac{p}{2p-1}}
\]
for all $x\in \operatorname{int}(M)$. In particular, there exists $\eps_0=\eps_0(n,p,K,D)>0$ such that if $\kappa(p,K)\le \eps_0$, then for all $x\in \operatorname{int}(M)$ and all $0<r_1\le r_2\le D$,
\begin{equation}\label{eq:appendix-doubling}
\vol(B_x(r_2))
\le
C_d(K,D,n)\left(\frac{r_2}{r_1}\right)^n \vol(B_x(r_1)),
\end{equation}
where one may take
\[
C_d(K,D,n)=
\begin{cases}
2, & K\ge 0,\\[1mm]
2e^{(n-1)\sqrt{|K|}D}, & K<0.
\end{cases}
\]
\end{theorem}

\begin{remark}
Although Theorem~\ref{thm:appendix-vol-comp} is stated for balls centered at interior points, which is sufficient for all applications below, such balls may meet
$\partial M$. In the presence of convex boundary, the proof
of the volume comparison theorem in \cite{petersen1997relative} still applies to intrinsic balls centered
at interior points, even when the ball meets the boundary, since the argument is radial and convexity
guarantees that minimizing geodesics issuing from the center remain in the interior of $M$.
\end{remark}

We now recall the analogue of Lemma~A.1 in \cite{RamosSeto}, which was proven in \cite[Lemma~(C)]{Gromov}.

\begin{lemma}\label{lem:appendix-A1}
Let $S$ be a hypersurface dividing $M$ into two parts $M_1$ and $M_2$. Let $W_i\subset M_i$
be Borel sets of positive measure. Then there exists $x_1\in W_1$ and a Borel set $W\subset W_2$ such that
\[
\vol(W)\ge \frac12 \vol(W_2),
\]
and for every $x\in W$, the unique minimizing geodesic from $x_1$ to $x$
meets $S$ at a point $z$ satisfying
\[
d(x_1,z)\ge d(x,z).
\]
\end{lemma}

\begin{lemma}\label{lem:appendix-A2}
Under the assumptions of Lemma~\ref{lem:appendix-A1}, let
\[
\sigma:=\sup_{x\in W} d(x_1,x),
\]
and let $S'\subset S$ be the set of intersection points with $S$ of minimizing geodesics joining $x_1$ to points of $W$. Then for every $p>\frac n2$,
\begin{equation}\label{eq:appendix-A2}
\vol(W)
\le
A_1(n,K,\sigma)\,\sigma\,\vol(S')
+
A_2(n,p,K,\sigma)\,\kappa(p,K)^{1/2}\,\vol(B_{x_1}(\sigma)),
\end{equation}
where
\[
A_1(n,K,\sigma)=
\begin{cases}
2^{n-1}, & K\ge 0,\\[1mm]
2^{n-1}\cosh^{n-1}\!\bigl(\sqrt{|K|}\sigma/2\bigr), & K<0,
\end{cases}
\]
and
\[
A_2(n,p,K,\sigma)=
\left(\frac{(n-1)(2p-1)}{2p-n}\right)^{1/2}A_1(n,K,\sigma).
\]
\end{lemma}

\begin{proof}
By Lemma~\ref{lem:appendix-A1}, for every $x\in W$ there is a unique minimizing
geodesic from $x_1$ to $x$, and its first intersection point $z$ with $S$ satisfies
\[
d(x_1,z)\ge d(x,z).
\]
By Lemma~\ref{lem:appendix-convex-geodesics}, these minimizing geodesics remain
in the interior of $M$. This is where the convexity of $\partial M$ is used.

Write the volume form in geodesic polar coordinates centered at $x_1$ as
\[
dv=\mathcal{A}(r,\theta)\,dr\,d\theta_{n-1}.
\]
Hence all points of $W$ can be parametrized in geodesic polar coordinates centered at $x_1$.
Let $\Gamma\subset S_{x_1}M$ be the set of unit vectors $\theta$ such that the radial geodesic
\[
\gamma_\theta(t)=\exp_{x_1}(t\theta)
\]
meets $W$. For $\theta\in \Gamma$, let:
\begin{itemize}
    \item $s_1(\theta)$ be the first parameter such that $\gamma_\theta(s_1(\theta))\in W$,
    \item $s_2(\theta)$ be the last parameter such that $\gamma_\theta(s_2(\theta))\in W$,
    \item $s(\theta)$ be the parameter such that $\gamma_\theta(s(\theta))\in S'$.
\end{itemize}
Then
\[
2s(\theta)\ge s_2(\theta)\ge s_1(\theta)\ge s(\theta).
\]

By Lemma~\ref{lem:appendix-local-jacobian}, the radial Jacobian satisfies
\[
\frac{\partial}{\partial r}\left(\frac{\mathcal{A}(r,\theta)}{\mathcal{A}_K(r)}\right)
\le
\psi(r,\theta)\frac{\mathcal{A}(r,\theta)}{\mathcal{A}_K(r)},
\]
where $\psi$ is controlled in $L^{2p}$ by $\kappa(p,K)^{1/2}$. Integrating from $s(\theta)$ to $s\in [s_1(\theta),s_2(\theta)]$ and using $s\le 2s(\theta)\le \sigma$, one gets
\[
\mathcal{A}(s,\theta)
\le
A_1(n,K,\sigma)\left(
\mathcal{A}(s(\theta),\theta)
+
\int_{s(\theta)}^s \psi(r,\theta)\mathcal{A}(r,\theta)\,dr
\right).
\]
Integrating over $W$ yields
\begin{align*}
\vol(W)
&=
\int_\Gamma \int_{s_1(\theta)}^{s_2(\theta)} \mathcal{A}(s,\theta)\,ds\,d\theta \\
&\le
A_1(n,K,\sigma)\,\sigma
\left(
\int_\Gamma \mathcal{A}(s(\theta),\theta)\,d\theta
+
\int_\Gamma \int_0^\sigma \psi(r,\theta)\mathcal{A}(r,\theta)\,dr\,d\theta
\right).
\end{align*}
As in \cite[Lemma~A.2]{RamosSeto}, the first term is bounded by $\vol(S')$, while the second is bounded by
\[
\left(\frac{(n-1)(2p-1)}{2p-n}\right)^{1/2}\kappa(p,K)^{1/2}\vol(B_{x_1}(\sigma)).
\]
This gives \eqref{eq:appendix-A2}.
\end{proof}

\begin{remark}
The role of the convexity of $\partial M$ is to guarantee that, by Lemma~\ref{lem:appendix-convex-geodesics}, minimizing geodesics joining interior points remain in the interior of $M$. This is what allows us in Lemma~\ref{lem:appendix-A2} to parametrize the relevant set by geodesic polar coordinates centered at $x_1$ and apply Jacobian comparison along minimizing geodesic rays.
\end{remark}

We now obtain the local relative isoperimetric estimate.

\begin{lemma}\label{lem:appendix-A3}
There exists $\eps=\eps(n,p,K,D)>0$ such that if $\kappa(p,K)\le \eps$, then for every ball $B_x(r)\subset M$ and every hypersurface $S$ dividing $B_x(r)$ into two subsets of equal volume,
\begin{equation}\label{eq:appendix-A3}
\vol(B_x(r))
\le
16A_1(n,K,D)\,r\,\vol(S\cap B_x(2r)).
\end{equation}
\end{lemma}

\begin{proof}
Let $M_1$ and $M_2$ be the two sides determined by $S$, and set
\[
W_i:=B_x(r)\cap M_i.
\]
Since $S$ divides $B_x(r)$ equally,
\[
\vol(B_x(r)\cap M_1)=\vol(B_x(r)\cap M_2)=\frac12\vol(B_x(r)).
\]
Apply Lemma~\ref{lem:appendix-A1} to $W_1$ and $W_2$. Then there exist $x_1\in W_1$ and $W\subset W_2$ such that
\[
\vol(W)\ge \frac12\vol(W_2)=\frac14\vol(B_x(r)).
\]
Hence
\[
\vol(B_x(r))\le 4\vol(W).
\]
Also $\sigma\le 2r$ and $S'\subset S\cap B_x(2r)$. By Lemma~\ref{lem:appendix-A2},
\[
\vol(B_x(r))
\le
8A_1(n,K,D)r\,\vol(S\cap B_x(2r))
+
4A_2(n,p,K,D)\kappa(p,K)^{1/2}\vol(B_x(2r)).
\]
Using the doubling estimate \eqref{eq:appendix-doubling}, the second term can be absorbed into the left-hand side if $\kappa(p,K)$ is sufficiently small. This gives \eqref{eq:appendix-A3}.
\end{proof}

For $1\le q\le \frac{n}{n-1}$, define the Neumann $q$-isoperimetric constant by
\[
I_N^q(M):=
\sup_S
\frac{\min\{\vol(M_1),\vol(M_2)\}^{1/q}}
{\vol(S)\vol(M)^{1-1/q}},
\]
where $S$ runs over all hypersurfaces dividing $M$ into two parts $M_1$ and $M_2$.

As in \cite{RamosSeto}, consider the following two properties:
\begin{itemize}
    \item[(i)] there is $C_d>0$ such that for all $x\in \operatorname{int}(M)$ and all $0<r_1\le r_2\le D$,
    \[
    \vol(B_x(r_2))
    \le
    C_d\left(\frac{r_2}{r_1}\right)^n\vol(B_x(r_1));
    \]
    \item[(ii)] there is $A>0$ such that if $S$ divides $B_x(r)$ equally, then
    \[
    \vol(B_x(r))\le Ar\,\vol(S\cap B_x(2r)).
    \]
\end{itemize}

\begin{proposition}\label{prop:appendix-A1}
If {\rm(i)} and {\rm(ii)} hold and $\diam(M)\le D$, then for all $1\le q\le \frac n{n-1}$,
\begin{equation*}%\label{eq:appendix-global-iso}
I_N^q(M)\le 10^{n/q}AC_dD.
\end{equation*}
\end{proposition}

\begin{proof}
Since $\partial M$ has measure zero, the Vitali covering argument may be carried out
using centers in $\operatorname{int}(M)$. The rest of the proof is exactly the same
as that of Proposition~A.1 in \cite{RamosSeto}.
\end{proof}

Combining Theorem~\ref{thm:appendix-vol-comp}, Lemma~\ref{lem:appendix-A3}, and Proposition~\ref{prop:appendix-A1}, we obtain a uniform bound for the Neumann isoperimetric constant.

\begin{corollary}\label{cor:appendix-uniform-iso}
Let $p>\frac n2$. There exists $\eps=\eps(n,p,K,D)>0$ such that if $\kappa(p,K)\le \eps$, then for every $1\le q\le \frac n{n-1}$,
\[
I_N^q(M)\le C(n,p,K,D).
\]
\end{corollary}

Note that when $q=1$, $I_N^1(M)$ is the reciprocal of the Cheeger constant. Therefore, using Cheeger's inequalities (see \cite{cheeger70}), we can obtain the following Neumann Poincar\'e inequality and rough lower bound for $\lambda_1$.

\begin{corollary}\label{cor:appendix-poincare}
Let $p>\frac n2$. There exists $\eps=\eps(n,p,K,D)>0$ such that if
$\kappa(p,K)\le \eps$, then there exists $C_P=C_P(n,p,K,D)>0$ such that for every
$u\in W^{1,2}(M)$,
\[
\|u-u_M\|_2^*\le C_P\|\nabla u\|_2^*.
\]
Equivalently,
\[
\lambda_1(M)\ge C_P^{-2}.
\]
\end{corollary}

Define the Neumann $q$-Sobolev constant by
\[
\SN_q(M):=
\sup_{u\in W^{1,1}(M)}
\frac{\displaystyle\left(\inf_{a\in\mathbb R}\fint_M|u-a|^q\,dv\right)^{1/q}}
{\displaystyle\fint_M|\nabla u|\,dv}.
\]

\begin{proposition}\label{prop:appendix-SN}
For every $q>0$,
\[
\min\{1,2^{1-1/q}\}I_N^q(M)
\le
\SN_q(M)
\le
\max\{1,2^{1-1/q}\}I_N^q(M).
\]
\end{proposition}

\begin{proof}
This is the same proof as Proposition~A.2 in \cite{RamosSeto}; the normalization by $\vol(M)$ cancels on both sides.
\end{proof}

The following is the averaged version of Proposition~A.3 in \cite{RamosSeto}.

\begin{proposition}\label{prop:appendix-L2-Sob}
Let $1<q<2$. There exist constants $C_1,C_2>0$, depending only on $q$, such that for every $u\in W^{1,2}(M)$,
\begin{equation*}%\label{eq:appendix-L2-sobolev-general}
\|\nabla u\|_2^{*\,2}
\ge
C_1\SN_q(M)^2
\left[
\|u\|_{\frac{2q}{2-q}}^{*\,2}
-
C_2\|u\|_2^{*\,2}
\right].
\end{equation*}
\end{proposition}

\begin{proof}
This is the same argument as Proposition~A.3 in \cite{RamosSeto}; see also \cite[Corollary~9.9]{li2012geometric}. Passing from the unnormalized to the normalized form only introduces powers of $\vol(M)$ which cancel after averaging.
\end{proof}

We are now ready to prove Proposition~\ref{prop:appendix-global-neumann-sobolev}.

\begin{proof}[Proof of Proposition~\ref{prop:appendix-global-neumann-sobolev}]
By Corollary~\ref{cor:appendix-uniform-iso} and Proposition~\ref{prop:appendix-SN}, there exists $\eps=\eps(n,p,K,D)>0$ such that if $\kappa(p,K)\le \eps$, then
\[
\SN_q(M)\le C(n,p,K,D)
\]
for every $1\le q\le \frac n{n-1}$.

Now choose
\[
q=\frac{2p}{2p-1}.
\]
Since $p>\frac n2$, we have $1<q<2$ and
\[
q=\frac{2p}{2p-1}\le \frac{n}{n-1}.
\]
Hence Proposition~\ref{prop:appendix-L2-Sob} applies for this choice of $q$. Moreover,
\[
\frac{2q}{2-q}=\frac{2p}{p-1}.
\]
Therefore Proposition~\ref{prop:appendix-L2-Sob} gives
\[
\|u\|_{\frac{2p}{p-1}}^{*\,2}
\le
C(n,p,K,D)\|\nabla u\|_2^{*\,2}
+
C'(p)\|u\|_2^{*\,2}.
\]
Applying this to $u-u_M$ and using the Neumann
Poincar\'e inequality coming from Corollary~\ref{cor:appendix-poincare}, we obtain
\[
\|u-u_M\|_{\frac{2p}{p-1}}^{*}
\le
C_S\|\nabla u\|_2^{*},
\]
which proves \eqref{eq:appendix-target-sobolev-mean-zero}.

Finally, write
\[
u=(u-u_M)+u_M.
\]
By the triangle inequality and the estimate
\[
|u_M|=\left|\fint_M u\,dv\right|\le \|u\|_2^{*},
\]
we get
\[
\|u\|_{\frac{2p}{p-1}}^{*}
\le
\|u-u_M\|_{\frac{2p}{p-1}}^{*}+\|u\|_2^{*}
\le
C_S\|\nabla u\|_2^{*}+\|u\|_2^{*},
\]
which proves \eqref{eq:appendix-target-sobolev}.
\end{proof}

\begin{remark}
The inequality \eqref{eq:appendix-target-sobolev} is the boundary analogue of the Sobolev inequality used in \cite[Proposition~2]{ramos2020zhong}. Once Proposition~\ref{prop:appendix-global-neumann-sobolev} is available, the Moser iteration argument used there carries over to the present setting.
\end{remark}

\bibliography{references}
\bibliographystyle{alpha}

\end{document}